\documentclass[12pt, reqno]{article}
\usepackage{latexsym}
\usepackage{color, hyperref, amsmath, amsthm, breqn}
\usepackage{amssymb}
\usepackage{amsfonts}

\newcommand{\anu}{\begin{enumerate}}
\newcommand{\pab}{\end{enumerate}}
\newcommand{\ee}{\end{equation}} 
\newcommand{\bb}{\begin{equation}} 
\newcommand{\Z}{\mathbb Z}

\newcommand{\Q}{\mathbb Q}

\newtheorem{lemma}{Lemma}[section]
\newtheorem{rem}{Remark}[section]
\newtheorem{theorem}{Theorem}[section]

\newtheorem{definition}{Definition}[section]

\newcommand{\md}{\!\!\!\!\!\pmod}

\title{ A complete classification of well-rounded real quadratic ideal lattices
}
\author{
Anitha Srinivasan\thanks{Saint Louis University-
Madrid Campus,
Avenida del Valle 34,
28003 Madrid, Spain.
email: rsrinivasan.anitha{@}gmail.com }, 
}
\date{}
\begin{document}

\maketitle

\setcounter{page}{1}

 \begin{abstract} 
We provide a complete classification of well-rounded ideal lattices
arising from real quadratic fields
$\mathbb{Q}(\sqrt{d})$ where  
$d$ is a field discriminant.
We show that the ideals that give rise to such lattices
are precisely the ones that correspond to divisors $a$ of $d$ that
satisfy
$\sqrt{\frac{d}{3}}<a<\sqrt{3d}.$
 \end{abstract} 
\section{Introduction}

Let $d$ be  field discriminant, that is $d\equiv 1\pmod 4$ is 
square free, or 
$d\equiv 0\pmod 4$ and $\frac{d}{4}$ is square free with  
$\frac{d}{4}\equiv 2$ or $3\pmod 4$.
A lattice in $\mathbb{R}^n$ is a $\mathbb{Z}$ module generated by $n$ linearly independent
vectors in $\mathbb{R}^n$.
 Recently ideal lattices have come into the forefront as 
a new tool for cryptography and coding theory, claiming to defeat the 
possible threats of a quantum attack on the hitherto established 
systems. 
  Non-zero vectors in a lattice with
least Euclidean norm play a key role as it is  
considered to be computationally difficult to find them. Therefore one is interested 
in a lattice that has a basis comprising of only such vectors.
We call such a lattice well-rounded. 
Our  focus will be on lattices that arise from ideals  in the full ring of integers
of a real quadratic field
$\Q(\sqrt{d})$. The reader is directed to  
\cite{DK}, \cite{FT} and \cite{FT2} for more information and results on the topic. 
In \cite{FT} the authors exhibit infinite families of real and imaginary
quadratic fields with ideals that give rise to well-rounded lattices.

In \cite{FT2} the authors consider divisors of $d$ that satisfy 
$\sqrt{\frac{d}{3}}<a<\sqrt{d}$ 
 and show that this condition is necessary and sufficient for the 
existence of well-rounded  ideals lattices for imaginary quadratic fields $\mathbb{Q}(\sqrt{-d})$. 
In the case of real quadratic fields, they pose a question  about  
 whether well-rounded ideals lattices can exist when this condition 
is not satisfied. Our main theorem below answers this question in the affirmative.
Indeed the necessary 
and sufficient condition given in the following theorem includes the condition given above. 
Note that an ideal is called well-rounded if the corresponding lattice  well-rounded.
\begin{theorem}
Let $K=\mathbb{Q}(\sqrt{d})$ be a real quadratic field where
$d$ is a field discriminant.
A primitive ideal $I$ in the ring of integers 
is well-rounded 
if and only if 
$I=a\mathbb Z+\frac{a-\sqrt{d}}{2}\mathbb Z$
and $a$ is a positive integer that satisfies 
$\sqrt{\frac{d}{3}}<a<\sqrt{3d}.$
\end{theorem}
\begin{rem}
The ideal $I$ in the above theorem has order dividing $2$ in the class group. 
Indeed, $d=a^2-4ac$ for some integer $c$, and hence 
$a=N(I)$, the norm of the ideal $I$
divides $d$.
If an ideal $I$ is not primitive, then 
$I=(\gamma )J$, where $J$ is a primitive ideal and $\gamma $ is a positive integer.
It is easy to see that $I$ is well-rounded if and only if $J$ is well-rounded. 
Note also that if $d=4d_1$ in Theorem 1.1, 
 then well-rounded ideals exist only
if $d_1\equiv 3\pmod 4$. This is because 
if $d$ is even, then $a=2a_1$ so that 
$d_1=a_1^2-2a_1c$, from which it follows that $d_1\equiv 3\pmod 4$. This 
completes the answer to 
 \cite[Question 2]{FT2} where the authors asked whether principal WR ideals
exist in the case when $d\not\equiv 1\pmod 4$. We have show above that
$d_1\equiv 2\pmod 4$ is not possible and 
 \cite[Proposition 4]{Ray}, 
states that there are infinitely many real quadratic fields
with WR principal ideals in the case when $d_1\equiv 3\pmod 4$.  

\end{rem}
\section{Binary quadratic forms, ideals and lattices}

\subsection{Forms}

A {\sl binary quadratic form}  is
a function $f(x, y)=ax^2+bxy+cy^2$, where $a, b,c $ are real numbers 
(called the coefficients of $f$)
and 
$d=b^2-4 a c$ is the {\it discriminant}.
A form is called {\it positive definite} if the discriminant 
is negative.  We consider only 
integral forms,  
 that 
is $a, b, c$ are integers. In the case when  
 $\gcd(a, b, c)=1$, the form is {\it primitive}. 
Often we will suppress the variables $x$ and $y$,
writing a form $f(x,y)$ as simply $f=(a, b, c)$.

Two forms $f$ and $f'$ are said to be {\it equivalent}, written as
$f\sim f'$,  if for some
$A=\begin{pmatrix} \alpha &\beta \\ \gamma & \delta \end{pmatrix}
\in SL_2(\mathbb Z)$ we have
$f'(x,y)=f(\alpha x+\beta y, \gamma x+\delta y)$. It is easy to see
that $\sim$ is an equivalence relation on the set of forms of
discriminant $d$. 

A form $f$ is said to {\it represent} an integer $m$ if there exist 
coprime integers $x$ and $y$, such that $f(x,y)=m$. 
Note that equivalent forms represent the same integers and hence
sometimes we refer to a class of forms $f$ that represents a
given integer. 

The set of equivalence classes of primitive 
forms is an abelian group called the {\it form class group}, with 
group law as composition given in Definition 2.1 in the next section. 

If $f=(a, b, c)$, then the form  $(a,-b,c)$ is the inverse of $f$. 
The {\it identity form} $e$ is defined as the form 
$(1,0,\frac{-d}{4})$
or $(1, 1, \frac{1-d}{4})$ 
depending on whether $d$ is even or odd
respectively. 
A useful fact is that any form that represents the integer $1$ is 
equivalent to the identity form. 

The {\it infimum} of a binary quadratic form $f$ 
is defined as 
$m(f)=\inf\{|f(x,y)|: x, y\in \Z\}$, where
$x, y$ are not both $0$.
Note that $m(f)=m(f^{-1})$ and $m(kf)=km(f)$ for any 
real number $k$.

A form $(a, b, c)$ of negative discriminant is {\it reduced} if 
$|b|\le a\le c$
where $b>0$ in the case when $|b|=a$ or $a=c$. 
 {\it Symmetric} forms satisfy $a=c$.


\subsection {Ideals}

Recall that 
$d\equiv 0, 1\hskip2mm\md 4$  denotes a field discriminant and
all ideals are in the ring of integers.
We present below a description of an ideal and the  rule 
for composing two primitive ideals. The reader may refer to
\cite[Sections 1.1 and 1.2]{Mo} for more information.

Let 
$$w=\begin{cases}
\frac{1+\sqrt{d}}{2}, & d\equiv 1\pmod 4 \\
\sqrt{\frac{d}{4}}, &  d\equiv 0\pmod 4
.
\end{cases}
$$
The ring of algebraic integers is the module with basis 
$[1, w]$. 
A primitive ideal $I$  can be written in the form 
\begin{equation}
I=a\mathbb Z+\frac{-b+\sqrt{d}}{2}\mathbb Z,
\end{equation}
where $a, b$ are integers such that $a>0$ is the norm of the ideal,
$0\le b<2a$ 
and $4a$ divides $b^2-d$. 

If $c=\frac{b^2-d}{4a}$ then $\gcd(a, b, c)=1$ as $d$ is a field discriminant
and so $(a, b, c)=ax^2+bxy+cy^2$ is a primitive form of discriminant $d$.
Also, if $I$ is not primitive then there is a primitive ideal $J$ 
such that $I=(\gamma)J$ for some integer $\gamma$.

%

In the following definition we 
present the formula for the product
of ideals which leads to composition of forms.

\begin{definition}\emph{(Composition law)}
Let $I_k=a_k\mathbb Z+\frac{-b_k+\sqrt{d}}{2}\mathbb Z,\hskip2mm
k=1, 2,$ be two primitive ideals. Let $f_1=(a_1, b_1, c_1) \text{ and }
f_2=(a_2, b_2, c_2)$ be the corresponding  binary quadratic forms
of discriminant $d$.
Let $g=gcd(a_1, a_2, (b_1+b_2)/2)$ and let $v_1, v_2, w $ be integers
such that $$v_1a_1+v_2a_2+w(b_1+b_2)/2=g.$$ If  $a_3$ and
$b_3$ are given by
\[
\begin{split} a_3&=\frac{a_1a_2}{g^2},\\
b_3&= b_2+2\,\frac{a_2}{g}\,\left(\frac{b_1-b_2}{2}\,\,v_2-c_2w\right) 
\mod {2a_3},
\end{split}
\]
then $I_1\cdot I_2$ is the ideal
$a_3\mathbb Z +\frac{-b_3+\sqrt{d}}{2}\mathbb Z$.
Also, the composition of the forms $(a_1, b_1, c_1)$ and
$(a_2, b_2, c_2)$ is the form $(a_3, b_3, c_3)$, where $c_3$ is
computed using the discriminant equation.
\end{definition}
Note that this gives the multiplication in the class group.
\subsection{Lemmas on binary quadratic forms}
The following lemma contains some elementary results on
binary quadratic forms that we use in the proof of the main theorem. 
\begin{lemma}\emph{
\begin{enumerate}
\item
The form $(a, b, c)$ is equivalent to the 
form $(a, b+2a\delta, a\delta^2+b\delta+c)$ for any integer $\delta$.
\item 
The form $(a, b, c)$ is equivalent to the form 
$(c, -b, a)$.
\item
If $f=(a, b, c)$ is a reduced form of negative discriminant,
then $a$ and $c$ are the two smallest  
 integers represented by $f$.
\end{enumerate}
}
\end{lemma} 
\begin{proof}
Parts are 1 and 2 are achieved by the transformation matrices
$\begin{pmatrix} 1 & \delta \\ 0 & 1 \end{pmatrix}
$
and 
$\begin{pmatrix} 0 & -1 \\ 1 & 0 \end{pmatrix}
$.
For the third part see
 \cite[Chapter 6, Section 8]{Ri}.
\end{proof}
In the following lemma we present an elementary fact on forms 
of order dividing two that is the principal tool in proving our main result.
\begin{lemma}
Let $f=(a, b, c)$ be a positive definite form of discriminant $d$ 
such that $f^2\sim e$. Then either 
$a$ divides $b$ or 
$\frac{a^2}{(\gcd(a, b))^2}\ge \frac{|d|}{4}$. 
\end{lemma}
\begin{proof}
Let the composition of $f$ with itself using the composition law (Definition 2.2) be 
$(A, B, C)$. Then 
$A=\left(\frac{a}{\gcd(a, b)}\right)^2$, and as equivalent forms represent the 
same integers, the identity form $e$ represents 
$A$. As  
$e$ is either $(1,0,\frac{-d}{4})$
or $(1, 1, \frac{1-d}{4})$ 
(depending on the parity of $d$), from Lemma 
2.1, part 3, 
if $A\ne 1$ 
then 
$A=\left(\frac{a}{\gcd(a, b)}\right)^2\ge \frac{|d|}{4}$. 
Note that $A=1$ corresponds to $a| b$.
\end{proof}
\subsection{Real quadratic ideal lattices}

Let $I$ be an ideal in the ring of integers of a real quadratic field
$\mathbb{Q}(\sqrt{d})$ with basis 
$\left[a, \frac{b-\sqrt{d}}{2}\right]$. The lattice 
associated to $I$ denoted by $L_I$ has a basis matrix 
$$A_I=\begin{pmatrix} a& \frac{b-\sqrt{d}}{2} \\
a& \frac{b+\sqrt{d}}{2}

\end{pmatrix}
,$$ 
 so that if 
$\begin{pmatrix}m\\n\end{pmatrix}\in \mathbb{Z}^2$,
the  elements of the lattice 
are given by  
${\bf x}=A_I
\begin{pmatrix}m\\n\end{pmatrix}
$,
with norm form 
\begin{equation}
Q_I=||{\bf x }||^2=
\begin{pmatrix}m& n\end{pmatrix}
A_I^T A_I
\begin{pmatrix}m\\n\end{pmatrix}
.
\label{eq: norm}
\end{equation}
The lattice is {\it well-rounded } or WR if there is a basis of elements
both of which have the shortest norm.
We call an ideal well-rounded if the corresponding lattice is WR.
\begin{lemma}
An ideal $I$ is WR if and only if its norm 
form 
$Q_I=t Q$, for some positive integer $t$ and $Q$ a primitive form
that is equivalent to a reduced symmetric form.
\end{lemma}
\begin{proof}
Let $I=(\gamma) J$ where $J$ is a primitive ideal.
Then 
$A_I=\gamma A_J$ and
from the norm equation (2),  
it follows that 
$Q_I=\gamma^2 Q_J$. Let $Q$ be the primitive form obtained 
by dividing 
the three coefficients of $Q_J$ by their  
 greatest common divisor $g$. Then   
$Q_I=\gamma^2 g Q$ and
the least positive integers represented satisfy
$m(Q_I)=\gamma^2 g m(Q)$. 
Clearly $I$ is WR if and only if $Q$ is WR.
By Lemma 2.1, part 3,  $Q$ is WR if and only if it
is equivalent to a reduced symmetric form and the result follows.
\end{proof}

\section{ Proof of Theorem 1.1}

Let 
 $I=a\mathbb Z+\frac{b-\sqrt{d}}{2}\mathbb Z$ be a primitive ideal.
Recall 
that $d$ is a field discriminant and hence $\gcd(a, b, c)=1$,
where $d=b^2-4ac$.
The corresponding lattice $L_I$ 
  has norm form (given by (2))
$$Q_I= 
2a^2m^2+2abmn+\frac{b^2+d}{2} n^2.$$ 
Note that $Q_I$ is a positive definite form of discriminant $-4a^2 d$.

Let $g=\gcd(2a^2,2ab,\frac{b^2+d}{2})$. 
As $d=b^2-4ac$ 
we have $g\le \gcd(a, b)$.
It follows that 
$$ Q=\frac{Q_I}{g}=\frac{2a^2}{g}m^2+\frac{2ab}{g}mn+\frac{b^2+d}{2g} n^2$$ 
is a primitive form 
of discriminant $-4\frac{a^2d}{g^2}$.

{\noindent{\bf Claim 1}} 
 $ Q^2 \sim e$ if and only if 
$a|b$.
\begin{proof}
It is easy to see from Definition 2.2 that for any primitive form
$(A, B, C)$, if $A|B$, then the composition of the form with itself 
gives the identity form (as $a_3=1$ in Definition 2.1 ). 
Therefore if $a|b$ then $ Q^2\sim e$. Now assume that 
$ Q^2\sim e$. 
We have $\gcd(\frac{2a^2}{g},\frac{2ab}{g})=
\frac{2a}{g}\gcd(a, b)$. 	 
By Lemma 2.2 either $a|b$ or  
$$\frac{\left(\frac{2a^2}{g}\right)^2}
{\left(\frac{2a}{g}\gcd(a, b)\right)^2}=
\frac{a^2}{\gcd(a, b)^2} 	 
\ge \frac{a^2}{g^2}d.$$ 	 
The above gives 
$d\le \frac{g^2}{\gcd(a, b)^2}\le 1$
(as $g\le \gcd(a, b)$), 
 which is not
possible. 

\end{proof}
%
%
{\noindent \bf Claim 2}  If $a|b$ then $ Q$
is equivalent to a reduced symmetric form if and only if
$a=b$ and $\sqrt{\frac{d}{3}}<a<\sqrt{3d}$. 

\begin{proof}
We first consider the case when  $d\equiv 1\pmod 4$. 
As $0\le b< 2a$, if $a|b$ it follows that $b=a$ (as $b\equiv d\pmod 4$ is odd) 
and $Q_I=(2a^2, 2a^2, \frac{a^2+d}{2})$.
Moreover $d=a^2-4ac$ and 
$\gcd(a, c)=1$ gives 
 $Q=(2a, 2a, a-2c)$. 

Using the equivalences given in parts 1 and 2 of Lemma 2.1, we have 
$Q\sim (a-2c, -2a, 2a)\sim (a-2c, 4c, a-2c)=f_0$.
Observe that $f_0$ is symmetric.
Also,  $f_0$ is reduced
if and only if  
$|4c|\le a-2c$. 
Using $c=\frac{a^2-d}{4a}$, it is easy to show that 
the condition $|4c|\le a-2c$ 
is equivalent to 
$\sqrt{d}<a<\sqrt{3d}$ when  
$c>0$ and to 
$\sqrt{\frac{d}{3}}<a<\sqrt{d}$ 
when $c<0$.


Now consider the case $d\equiv 0\pmod 4$.  
Let $d=4d_1$ and $b=2b_1$.
In this case $a|b$ gives $b=a$ or $b=0$. 
We look first at the case when $a=b$.
As $d_1=b_1^2-2b_1c$ and $\gcd(a, b, c)=1$, it 
follows that $b_1$ and $c$ are odd. 
Therefore
$Q_I=(2a^2, 2ab, 2(b_1^2+d_1))
=2(a^2, ab, 2b_1^2-ac)
=
8b_1(b_1, b_1, \frac{b_1-c}{2})$ 
and thus 
$Q=(b_1, b_1, \frac{b_1-c}{2})$.
We have the equivalences (using Lemma 2.3) 
$$
\left(b_1, b_1, \frac{b_1-c}{2}\right)
\sim \left(\frac{b_1-c}{2}, -b_1, b_1\right)
\sim \left(\frac{b_1-c}{2}, -b_1+2\frac{b_1-c}{2},
 \frac{b_1-c}{2}\right)
$$
$$\sim \left(\frac{b_1-c}{2}, c,
 \frac{b_1-c}{2}\right)=f_0.$$

The form $f_0$ is reduced iff
$|c|\le \frac{b_1-c}{2}$. 
 In an identical fashion to the case above, noting that $a=b=2b_1$ we obtain 
that $f_0$ is reduced if and only if
$\frac{\sqrt{d}}{3}<a<\sqrt{3d}$.
To complete the proof of Claim 2 it remains to consider the case $b=0$. 
We have $d_1=-ac$ and hence
$Q_I=(2a^2, 0, -2ac)$ 
and $Q=(a, 0, -c)$.
One of $(a, 0, -c)$ or $(-c, 0, a)$
is reduced and clearly neither is symmetric 
as $a\ne -c$ ($\gcd(a,c)=1$).
\end{proof}
From Lemma 2.3 the ideal 
$I$ is WR if and only if
$ Q$ is equivalent to a reduced symmetric form. 
From Claims 1 and 2
it follows that 
$I$ is WR if and only if
$a=b$ and 
$\sqrt{\frac{d}{3}}<a<\sqrt{3d}$. 
\qed

\end{document}